\documentclass[final,3p]{elsarticle}
\usepackage[latin1]{inputenc}
 \usepackage{graphics}
 \usepackage{graphicx}
 \usepackage{epsfig}
\usepackage{amssymb}
 \usepackage{amsthm}
 \usepackage{lineno}
 \usepackage{amsmath}
\usepackage{color}
   \numberwithin{equation}{section}
\usepackage{mathrsfs}

\journal{ } % ÐÞ¸ÄÔÓÖ¾Ãû

\newtheorem{thm}{Theorem}[section]

\newtheorem{lem}[thm]{Lemma}

\newtheorem{rem}[thm]{Remark}
\newtheorem{ex}[thm]{Example}
\newtheorem{pro}[thm]{Property}

\biboptions{sort&compress,square}
\allowdisplaybreaks
\begin{document}
\begin{frontmatter}
\author{Hongfeng Li}
\ead{lihf728@nenu.edu.cn}
\author{Yong Wang\corref{cor2}}
\ead{wangy581@nenu.edu.cn}
\cortext[cor2]{Corresponding author.}

\address{School of Mathematics and Statistics, Northeast Normal University,
Changchun, 130024, China}

\title{The spectral Einstein functional for the nonminimal de Rham-Hodge operator}
\begin{abstract}
In this paper, we give the definitions of the non-self-adjoint spectral triple and its spectral Einstein functional. We compute the spectral Einstein functional associated with the nonminimal de Rham-Hodge operator on even-dimensional compact manifolds without boundary. Finally, several examples of the non-self-adjoint spectral triple are listed.
\end{abstract}
\end{frontmatter}
\section{Introduction}
 Until now, many geometers have studied the noncommutative residue, which is of great importance to the study of noncommutative geometry. Connes showed that the noncommutative residue on a compact manifold $M$ coincided with the Dixmier's trace on pseudo-differential operators of order $-{\rm {dim}}M$ \cite{Co2}. Hence, the noncommutative residue can be used as integral of noncommutative geometry and become an important tool of noncommutative geometry. In \cite{Co1}, Connes used the noncommutative residue to derive a conformal 4-dimensional Polyakov action analogy.
A few years ago, Connes made a challenging observation that the noncommutative residue of the square of the inverse of the Dirac operator was proportional to the Einstein-Hilbert action, which is called the Kastler-Kalau-Walze type theorem. Kastler \cite{Ka} gave a brute-force proof of this theorem. In \cite{KW}, Kalau and Walze proved this theorem in the normal coordinates system simultaneously. Based on the theory of the noncommutative residue proposed by Wodzicki, Fedosov, et al. \cite{FGLS} constructed the noncommutative residue of the classical elemental algebra of the Boutte de Monvel calculus on compact manifolds of dimension $n>2$.

Using elliptic pseudo-differential operators and the noncommutative residue is a natural way to study the spectral Einstein functional and the operator-theoretic interpretation of the gravitational action on bounded manifolds. Concerning the Dirac operator and the signature operator, Wang carried out the computation of noncommutative residues and succeeded in proving the Kastler-Kalau-Walze type theorem for manifolds with boundary \cite{Wa1, Wa3, Wa4}. In \cite{FGV2}, Figueroa, et al. introduced a noncommutative integral based on the noncommutative residue.  Pf${\mathrm{\ddot{a}}}$ffle and Stephan \cite{pf1} considered orthogonal connections with arbitrary torsion on
compact Riemannian manifolds and computed the spectral action. In \cite{DL}, Dabrowski, Sitarz and Zalecki defined the spectral Einstein functional for a general spectral triple and for the noncommutative torus, they computed the spectral Einstein functional.
 In \cite{WWw}, Wang, et al. gave some new spectral functionals which are the extension of spectral functionals to the noncommutative realm with torsion, and related them to the noncommutative residue for manifolds with boundary about Dirac operators with torsion.
 In \cite{DL2}, Dabrowski, Sitarz and Zalecki examined the metric and Einstein bilinear functionals of differential forms for the Hodge-Dirac operator $d+d^*$ on an oriented, closed, even-dimensional Riemannian manifold. Hong and Wang computed the spectral Einstein functional associated with the Dirac operator with torsion
 on even-dimensional spin manifolds without boundary in \cite{Hj}. Wu and Wang computed the spectral Einstein functional for the Witten deformation $d+d^*+\widehat{c}(V)$ on even-dimensional Riemannian manifolds without boundary \cite{Wu2}. In \cite{Wu1}, Wu and Wang obtained the Lichnerowicz type formula
for the Witten deformation of the nonminimal de Rham-Hodge operator, and the gravitational action in
the case of $n$-dimensional compact manifolds without boundary. Finally, the authors gave the proof of Kastler-Kalau-Walze type theorems for the Witten deformation of the nonminimal de Rham-Hodge operator on $4,6$-dimensional oriented compact manifolds with boundary. In \cite{WWS}, the authors proved Kastler-Kalau-Walze type
theorems associated with nonminimal de Rham-Hodge operators on compact manifolds with boundary.

{\bf The motivation of this paper} is to compute the spectral Einstein functional for the nonminimal de Rham-Hodge operator $\widetilde{D}=a_0d+b_0\delta$ on even-dimensional Riemannian manifolds without boundary, where $d,\delta,a_0,b_0$ are defined in Section \ref{section:2}. Our main theorem is as follows.
 \begin{thm}\label{1thm}
 	Let $M$ be an $n=2m$ dimensional ($n\geq 3$) Riemannian manifold, for the nonminimal de Rham-Hodge operator $\widetilde{D}$, the spectral metric functional $\mathscr{M}_{\widetilde{D}}$ and the spectral Einstein functional $\mathscr{N}_{\widetilde{D}}$ are equal to
 	
 \begin{align}
 	\mathscr{M}_{\widetilde{D}}&=\mathrm{Wres}\bigg(\widetilde{c}(u)\widetilde{c}(v){\widetilde{D}^{-2m}}\bigg)=-2^{2m} \frac{2 \pi^{m}}{\Gamma\left(m\right)}\int_{M} (a_0b_0)^{-m+1}g(u,v)d{\rm Vol}_M;\nonumber\\
 	\mathscr{N}_{\widetilde{D}}&=\mathrm{Wres}\bigg(\widetilde{c}(u)\big(\widetilde{D}\widetilde{c}(v)+\widetilde{c}(v)\widetilde{D}\big)\widetilde{D}^{-2m+1}\bigg)
=-2^{2m} \frac{2 \pi^{m}}{\Gamma\left(m\right)}\int_{M}\frac{(a_0b_0)^{-m+2}}{6}\mathbb{G}(u,v)d{\rm Vol}_M,\nonumber
 \end{align}
 where  $g(u,v)=\sum_{a,b=1}^{n}u_{a} v_{b} $, $ \mathbb{G}(u,v)=\operatorname{Ric}(u,v)-\frac{1}{2} s g(u,v) $, $\widetilde{c}(u)=\sum_{\eta=1}^{n} u_{\eta}\widetilde{c}(e_{\eta}), \widetilde{c}(v)=\sum_{\gamma=1}^{n} v_{\gamma}\widetilde{c}(e_{\gamma})$ and $\widetilde{c}(e_j)=a_0\epsilon(e_j^*)-b_0\iota(e_j)$.
 \end{thm}
 
\begin{rem}
The Einstein functional is still recovered in the non-self-adjoint case.
\end{rem}

The paper is organized in the following way. In Section \ref{section:2}, we introduce some notations about Clifford action and the symbols of the generalized laplacian for the nonminimal de Rham-Hodge operator. In Section \ref{section:3}, using the residue for a differential operator of Laplace type ${\rm Wres}(P):=\int_{S^*M}{\rm tr}(\sigma_{-n}^P)(x,\xi)$ and the composition formula of pseudo-differential operators, we obtain the metric functional $\mathscr{M}_{\widetilde{D}}$ and the spectral Einstein functional $\mathscr{N}_{\widetilde{D}}$ for the nonminimal de Rham-Hodge operator on even-dimensional Riemannian manifolds without boundary. In Section \ref{section:4}, we introduce several examples of non-self-adjoint spectral triple, all of which are important examples of the noncommutative geometry.

\section{Clifford action and the nonminimal de Rham-Hodge operator}
\label{section:2}
Firstly, we introduce some notations about Clifford action. Let $M$ be an $n$-dimensional ($n\geq 3$) oriented compact Riemannian manifold with a Riemannian metric $g^{TM}$.\\
\indent Let $\nabla^L$ be the Levi-Civita connection about $g^{TM}$. In the
fixed orthonormal frame $\{e_1,\cdots,e_n\}$, the connection matrix $(\omega_{s,t})$ is defined by
\begin{equation}
\label{eq1}
\nabla^L(e_1,\cdots,e_n)= (e_1,\cdots,e_n)(\omega_{s,t}).\nonumber
\end{equation}
\indent Let $\epsilon (e_j^*)$,~$\iota (e_j^*)$ be the exterior and interior multiplications respectively, where $e_j^*=g^{TM}(e_j,\cdot)$. And $c(e_j)$ be the Clifford action. Write
\begin{equation}
\label{eq2}
\widehat{c}(e_j)=\epsilon (e_j^* )+\iota
(e_j^*);~~
c(e_j)=\epsilon (e_j^* )-\iota (e_j^* ),\nonumber
\end{equation}
which satisfies
\begin{align}
\label{ali1}
&\widehat{c}(e_i)\widehat{c}(e_j)+\widehat{c}(e_j)\widehat{c}(e_i)=2g^{TM}(e_i,e_j);~~\nonumber\\
&c(e_i)c(e_j)+c(e_j)c(e_i)=-2g^{TM}(e_i,e_j);~~\nonumber\\
&c(e_i)\widehat{c}(e_j)+\widehat{c}(e_j)c(e_i)=0.
\end{align}

Next, we recall the nonminimal de Rham-Hodge operator. Following \cite{Wu1}, the de Rham derivative $d$ is a differential operator on $C^{\infty}(M,\wedge^*T^*M)$ and the de Rham coderivative $\delta=d^*$. Let $\widetilde{c}(e_j)=a_0\epsilon(e_j^*)-b_0\iota(e_j)$, where $a_0, b_0$ are constants, then the nonminimal de Rham-Hodge operator is defined
\begin{align}
\widetilde{D}=&a_0d+b_0\delta\nonumber\\
=&\sum^n_{i=1}\widetilde{c}(e_i)\bigg[e_i+\frac{1}{4}\sum_{s,t}\omega_{s,t}
(e_i)[\widehat{c}(e_s)\widehat{c}(e_t)
-c(e_s)c(e_t)]\bigg].
\end{align}

For a differential operator of Laplace type $P$, it has locally the form
\begin{equation}\label{p}
	P=-(g^{ij}\partial_i\partial_j+A^i\partial_i+B),
\end{equation}
where $\partial_{i}$  is a natural local frame on $TM,$ $(g^{ij})_{1\leq i,j\leq n}$ is the inverse matrix associated to the metric
matrix  $(g_{ij})_{1\leq i,j\leq n}$ on $M,$ $A^{i}$ and $B$ are smooth sections of $\textrm{End}(N)$ on $M$ (endomorphism).
If $P$ satisfies the form \eqref{p}, then there is a unique
connection $\nabla$ on $N$ and a unique endomorphism $\widetilde{E}$ such that
\begin{equation}
	P=-[g^{ij}(\nabla_{\partial_{i}}\nabla_{\partial_{j}}- \nabla_{\nabla^{L}_{\partial_{i}}\partial_{j}})+\widetilde{E}].
\end{equation}

Moreover
(with local frames of $T^{*}M$ and $N$), $\nabla_{\partial_{i}}=\partial_{i}+\omega_{i} $
and $\widetilde{E}$ are related to $g^{ij}$, $A^{i}$ and $B$ through
\begin{eqnarray}
	&&\omega_{i}=\frac{1}{2}g_{ij}\big(A^{i}+g^{kl}\Gamma_{ kl}^{j} \texttt{id}\big);\nonumber\\
	&&\widetilde{E}=B-g^{ij}\big(\partial_{i}(\omega_{j})+\omega_{i}\omega_{j}-\omega_{k}\Gamma_{ ij}^{k} \big),\nonumber
\end{eqnarray}
where $\Gamma_{ kl}^{j}$ is the  Christoffel coefficient of $\nabla^{L}$.

By $d^2=\delta^2=0$, the following identity holds:
\begin{align}
\widetilde{D}^{2}=&a_0^2d^2+b_0^2\delta^2+a_0b_0(d\delta+\delta d)\nonumber\\
=&a_0b_0(d\delta+\delta d)\nonumber\\
=&a_0b_0(d+\delta)^2\nonumber\\
:=&a_0b_0\widetilde{D}_0^{2}.
\end{align}

From \cite{DL} ,we get
\begin{align} \label{dt2}
\widetilde{D}_0
^{2}=-g^{ab}(\nabla_{{\partial}_a}\nabla_{{\partial}_b}-\nabla_{{\nabla_{{\partial}_a}^{L}}\partial_b})+E,\nonumber
\end{align}
where $E=-\widetilde{E}, \nabla_{{\partial}_a}={\partial}_a-\widetilde{T}_a$.

\indent By \cite{Y}, we have
\begin{equation}\label{eq8}
(d+\delta)^2
=-\triangle+\frac{1}{8}\sum_{ijkl}R_{ijkl}\widehat{c}(e_i)\widehat{c}(e_j)c(e_k)c(e_l)+\frac{1}{4}s;\nonumber
\end{equation}
\begin{equation}
\label{eq9}
-\triangle=-g^{ij}(\nabla^{\Lambda^*(T^*M)}_{i}\nabla^{\Lambda^*(T^*M)}_{j}-\Gamma_{ij}^{k}\nabla^{\Lambda^*(T^*M)}_{k}),\nonumber
\end{equation}
where $R$ (resp. $s$) denote the Riemannian curvature (resp. scalar curvature).

Define $\omega_{s,t}
(e_i)=-\langle \nabla_{e_i}^{L}e_{s}, e_{t}\rangle $, we get
\begin{align} \label{ta1}
	\widetilde{T}_a&=-\frac{1}{4}\sum_{s, t=1}^{n}\langle \nabla_{{\partial}_a}^{L}e_{s}, e_{t}\rangle c(e_{s})c(e_{t})+\frac{1}{4}\sum_{s, t=1}^{n}\langle \nabla_{{\partial}_a}^{L}e_{s}, e_{t}\rangle \widehat{c}(e_{s})\widehat{c}(e_{t});\nonumber\\
E&=\frac{1}{8}\sum_{i,j,k,l=1}^nR_{ijkl}\widehat{c}(e_i)\widehat{c}(e_j)c(e_k)c(e_l)+\frac{1}{4}s.
\end{align}

In normal coordinates, $\widetilde{T}_a$ is expanded near $x=0$ by Taylor expansion. That is
\begin{align}
	\widetilde{T}_a={T}_a+{T}_{ab}x^b+O(x^2).\nonumber
\end{align}

By ${\partial}_{l}\langle \nabla_{{\partial}_a}^{L}e_{s}, e_{t}\rangle(x_0)=\frac{1}{2}{\rm R}_{lats}(x_0)$, we get
\begin{align} \label{at0}
	{T}_a=0,
\end{align}
and
\begin{align} \label{ta0}
	{T}_{ab}=-\frac{1}{8}\sum_{a,b,s,t=1}^{2m}R_{bats}(x_0) c(e_{s})c(e_{t})+\frac{1}{8}\sum_{a,b,s,t=1}^{2m}R_{bats}(x_0) \widehat{c}(e_{s})\widehat{c}(e_{t}).
\end{align}

 \begin{lem}\label{lem1}\cite{DL}
	The leading symbols of the generalized laplacian $\Delta_{T,E}^{-m}$ are as follows:
	\begin{align}
		\sigma_{-2 m}(\Delta_{T,E}^{-m})=&\|\xi\|^{-2 m-2}\sum_{a,b,j,k=1}^{2m}\left(\delta_{a b}-\frac{m}{3} R_{a j b k} x^{j} x^{k}\right) \xi_{a} \xi_{b}+O\left(\mathbf{x}^{3}\right) ;\nonumber\\
		\sigma_{-2m-1}(\Delta_{T,E}^{-m})=& \frac{-2 m i}{3}\|\xi\|^{-2 m-2} \sum_{a,k=1}^{2m}\operatorname{Ric}_{a k} x^{k} \xi_{a}-2 m i\|\xi\|^{-2 m-2}\sum_{a,b=1}^{2m}\left(T_{a} \xi_{a}+T_{a b} x^{b} \xi_{a}\right)+O(\mathbf{x^2}) ;\nonumber\\
		\sigma_{-2m-2}(\Delta_{T,E}^{-m})=& \frac{m(m+1)}{3}\|\xi\|^{-2 m-4} \sum_{a,b=1}^{2m}\operatorname{Ric}_{a b} \xi_{a} \xi_{b}\nonumber \\
		&-2 m(m+1)\|\xi\|^{-2 m-4}\sum_{a,b=1}^{2m} T_{a} T_{b} \xi_{a} \xi_{b}+m\sum_{a,b=1}^{2m}\left(T_{a} T_{a}-T_{a a}\right)\|\xi\|^{-2 m-2}\nonumber \\
		&+2 m(m+1)\|\xi\|^{-2 m-4} \sum_{a,b=1}^{2m}T_{a b} \xi_{a} \xi_{b}-mE \|\xi\|^{-2m-2}+O(\mathbf{x}) ,\nonumber
	\end{align}
where $Ric$ denotes Ricci curvature.
\end{lem}

By \eqref{ta1}-\eqref{ta0} and Lemma \ref{lem1}, we get the following lemma.

\begin{lem}\label{lem2}
	General dimensional symbols of the operator $\widetilde{D}_0$ are given:
	\begin{align}
		\sigma_{-2 m}(\widetilde{D}_0^{-2m})=&\|\xi\|^{-2 m-2}\sum_{a,b,j,k=1}^{2m} \left(\delta_{a b}-\frac{m}{3} R_{a j b k} x^{j} x^{k}\right) \xi_{a} \xi_{b}+O\left(\mathbf{x}^{3}\right); \nonumber\\
		\sigma_{-2m-1}(\widetilde{D}_0^{-2m})=& -\frac{2 mi}{3}\|\xi\|^{-2 m-2}\sum_{a,b=1}^{2m}  \operatorname{Ric}_{a b} x^{b} \xi_{a}\nonumber\\
		&+\frac{ m i}{4}\|\xi\|^{-2 m-2} \sum_{a,b,s,t=1}^{2m} \operatorname{R}_{b a t s}(x_0) c(e_s) c(e_t) x^{b} \xi_{a}\nonumber\\
		&-\frac{ m i}{4}\|\xi\|^{-2 m-2} \sum_{a,b,s,t=1}^{2m} \operatorname{R}_{b a t s}(x_0) \widehat{c}(e_s) \widehat{c}(e_t) x^{b} \xi_{a} +O\left(\mathbf{x}^{2}\right);\nonumber\\
		\sigma_{-2m-2}(\widetilde{D}_0^{-2m})=& \frac{m(m+1)}{3}\|\xi\|^{-2 m-4}\sum_{a,b=1}^{2m} \operatorname{Ric}_{a b} \xi_{a} \xi_{b} \nonumber\\
		&-\frac{ m (m+1)}{4}\|\xi\|^{-2 m-4} \sum_{a,b,s,t=1}^{2m}\operatorname{R}_{b a t s}(x_0) c(e_s) c(e_t) \xi_{a} \xi_{b}\nonumber\\
&+\frac{ m (m+1)}{4}\|\xi\|^{-2 m-4} \sum_{a,b,s,t=1}^{2m}\operatorname{R}_{b a t s}(x_0) \widehat{c}(e_s) \widehat{c}(e_t) \xi_{a} \xi_{b}\nonumber\\
		&-\frac{1}{8}m\|\xi\|^{-2 m-2}\sum_{ijkl}R_{ijkl}\widehat{c}(e_i) \widehat{c}(e_j)c(e_k) c(e_l)\nonumber\\
&-\frac{1}{4}m \|\xi\|^{-2 m-2}s+O\left(\mathbf{x}\right),\nonumber
	\end{align}
where $i$ in $mi$ represents the imaginary unit, and $i$ in other subscript positions represents the $i$-th. 
\end{lem}

 \section{The spectral Einstein functional for the nonminimal de Rham-Hodge operator}
 \label{section:3}
In this section, we want to obtain the spectral metric functional and the spectral Einstein functional for the nonminimal de Rham-Hodge operator, which generates a non-self-adjoint spectral triple and its spectral metric functional and spectral Einstein functional are defined in Theorem \ref{thm}.

For a pseudo-differential operator  $P$, acting on sections of a vector bundle over an $n$-dimensional compact Riemannian manifold $M$ ($n$ is even), the analogue of the volume element in noncommutative geometry is the operator  $D^{-n}=: d s^{n} $. And pertinent operators are realized as pseudo-differential operators on the spaces of sections. Extending previous definitions by Connes \cite{co5}, a noncommutative integral was introduced in \cite{FGV2} based on the noncommutative residue \cite{wo2}, combine (1.4) in \cite{co4} and \cite{Ka}, using the definition of the residue:
\begin{align}\label{wers}
	\int P d s^{n}:=\operatorname{Wres} P D^{-n}:=\int_{S^{*} M} \operatorname{tr}\left[\sigma_{-n}\left(P D^{-n}\right)\right](x, \xi),
\end{align}
where  $\sigma_{-n}\left(P D^{-n}\right) $ denotes the  $(-n)$th order piece of the complete symbols of  $P D^{-n} $,  $\operatorname{tr}$  as shorthand of trace.

Firstly, we review here technical tool of the computation, which are the integrals of polynomial functions over the unit spheres. By (32) in \cite{B1}, we define
\begin{align}
I_{S_n}^{\gamma_1\cdot\cdot\cdot\gamma_{2\bar{n}+2}}=\int_{|x|=1}d^nxx^{\gamma_1}\cdot\cdot\cdot x^{\gamma_{2\bar{n}+2}},
\end{align}
i.e. the monomial integrals over a unit sphere.
Then by Proposition A.2. in \cite{B1},  polynomial integrals over higher spheres in the $n$-dimesional case are given
\begin{align}
I_{S_n}^{\gamma_1\cdot\cdot\cdot\gamma_{2\bar{n}+2}}=\frac{1}{2\bar{n}+n}[\delta^{\gamma_1\gamma_2}I_{S_n}^{\gamma_3\cdot\cdot\cdot\gamma_{2\bar{n}+2}}+\cdot\cdot\cdot+\delta^{\gamma_1\gamma_{2\bar{n}+1}}I_{S_n}^{\gamma_2\cdot\cdot\cdot\gamma_{2\bar{n}+1}}],
\end{align}
where $S_n\equiv S^{n-1}$ in $\mathbb{R}^n$.

For $\bar{n}=0$, we have $I^0={\rm Vol}(S^{n-1})$=$\frac{2\pi^{\frac{n}{2}}}{\Gamma(\frac{n}{2})}$, we immediately get
\begin{align}
I_{S_n}^{\gamma_1\gamma_2}&=\frac{1}{n}{\rm Vol}(S^{n-1})\delta^{\gamma_1\gamma_2};\nonumber\\
I_{S_n}^{\gamma_1\gamma_2\gamma_3\gamma_4}&=\frac{1}{n(n+2)}{\rm Vol}(S^{n-1})[\delta^{\gamma_1\gamma_2}\delta^{\gamma_3\gamma_4}+\delta^{\gamma_1\gamma_3}\delta^{\gamma_2\gamma_4}+\delta^{\gamma_1\gamma_4}\delta^{\gamma_2\gamma_3}].
\end{align}

 \begin{thm}\label{thm}
 	Let $M$ be an $n=2m$ dimensional ($n\geq 3$) Riemannian manifold, for the nonminimal de Rham-Hodge operator $\widetilde{D}$, the spectral metric functional $\mathscr{M}_{\widetilde{D}}$ and the spectral Einstein functional $\mathscr{N}_{\widetilde{D}}$ are equal to
 	
 \begin{align}
 	\mathscr{M}_{\widetilde{D}}&:=\mathrm{Wres}\bigg(\widetilde{c}(u)\widetilde{c}(v){\widetilde{D}^{-2m}}\bigg)=-2^{2m} \frac{2 \pi^{m}}{\Gamma\left(m\right)}\int_{M} (a_0b_0)^{-m+1}g(u,v)d{\rm Vol}_M;\nonumber\\
 	\mathscr{N}_{\widetilde{D}}&:=\mathrm{Wres}\bigg(\widetilde{c}(u)\big(\widetilde{D}\widetilde{c}(v)+\widetilde{c}(v)\widetilde{D}\big)\widetilde{D}^{-2m+1}\bigg)
 =-2^{2m} \frac{2 \pi^{m}}{\Gamma\left(m\right)}\int_{M}\frac{(a_0b_0)^{-m+2}}{6}\mathbb{G}(u,v)d{\rm Vol}_M,\nonumber
 \end{align}
 where  $g(u,v)=\sum_{a,b=1}^{n}u_{a} v_{b} $, $ \mathbb{G}(u,v)=\operatorname{Ric}(u,v)-\frac{1}{2} s g(u,v) $, $\widetilde{c}(u)=\sum_{\eta=1}^{n} u_{\eta}\widetilde{c}(e_{\eta}), \widetilde{c}(v)=\sum_{\gamma=1}^{n} v_{\gamma}\widetilde{c}(e_{\gamma})$ and $\widetilde{c}(e_j)=a_0\epsilon(e_j^*)-b_0\iota(e_j)$.
 \end{thm}

\begin{proof}
The proof of the $\mathscr{M}_{\widetilde{D}}$ formula for the measure function is obvious. Before we prove that the spectral Einstein function $\mathscr{N}_{\widetilde{D}}$, splitting it into two parts as follows:
\begin{align}
\mathscr{N}_{\widetilde{D}}&=\mathrm{Wres}\bigg(\widetilde{c}(u)\widetilde{D}\widetilde{c}(v) \widetilde{D}\widetilde{D}^{-2m}\bigg)+\mathrm{Wres}\bigg(\widetilde{c}(u)\widetilde{c}(v){\widetilde{D}^{-2m+2}}\bigg)\nonumber\\
&=(a_0b_0)^{-m}\mathrm{Wres}\bigg(\widetilde{c}(u)\widetilde{D}\widetilde{c}(v)\widetilde{D}\widetilde{D}_0^{-2m}\bigg)+(a_0b_0)^{-m+1}\mathrm{Wres}\bigg(\widetilde{c}(u)\widetilde{c}(v)\widetilde{D}_0^{-2m+2}\bigg)\nonumber\\
&:=\mathscr{N}_{1}+\mathscr{N}_{2}.\nonumber
\end{align}

{\bf Part I)} $\mathscr{N}_{1}=(a_0b_0)^{-m}\mathrm{Wres}\bigg(\widetilde{c}(u)\widetilde{D}\widetilde{c}(v)\widetilde{D}\widetilde{D}_0^{-2m}\bigg)$.

Let $\widetilde{c}(u)\widetilde{D}:=\mathcal{P},\widetilde{c}(v)\widetilde{D}:=\mathcal{Q} .$
By \eqref{wers}, we need to compute  $\int_{\wedge^*T^*M} \operatorname{tr}\big[\sigma_{-2 m}(\mathcal{P} \mathcal{Q} \widetilde{D}_0
^{-2 m})\big](x, \xi) $. Based on the algorithm yielding the principal symbol of a product of pseudo-differential operators in terms of the principal symbols of the factors, we have
\begin{align}\label{ABD}
	&\sigma_{-2 m}\left(\mathcal{P} \mathcal{Q} \widetilde{D}_0^{-2 m}\right)  \nonumber\\
     =&\left\{\sum_{|\alpha|=0}^{\infty} \frac{(-i)^{|\alpha|}}{\alpha!} \partial_{\xi}^{\alpha}\big[\sigma(\mathcal{P} \mathcal{Q})\big] \partial_{x}^{\alpha}\big[\sigma(\widetilde{D}_0^{-2 m})\big]\right\}_{-2 m} \nonumber\\
	 =&\sigma_{0}(\mathcal{P} \mathcal{Q}) \sigma_{-2 m}(\widetilde{D}_0^{-2 m})+\sigma_{1}(\mathcal{P} \mathcal{Q}) \sigma_{-2 m-1}(\widetilde{D}_0^{-2 m})+\sigma_{2}(\mathcal{P} \mathcal{Q}) \sigma_{-2 m-2}(\widetilde{D}_0^{-2 m}) \nonumber\\
	& +(-i) \sum_{j=1}^{2m} \partial_{\xi_{j}}\big[\sigma_{2}(\mathcal{P} \mathcal{Q})\big] \partial_{x_{j}}\big[\sigma_{-2 m-1}(\widetilde{D}_0^{-2 m})\big]+(-i) \sum_{j=1}^{2m} \partial_{\xi_{j}}\big[\sigma_{1}(\mathcal{P} \mathcal{Q})\big] \partial_{x_{j}}\big[\sigma_{-2 m}(\widetilde{D}_0^{-2 m})\big] \nonumber\\
	& -\frac{1}{2} \sum_{j,l=1}^{2m} \partial_{\xi_{j}} \partial_{\xi_{l}}\big[\sigma_{2}(\mathcal{P} \mathcal{Q})\big] \partial_{x_{j}} \partial_{x_{l}}\big[\sigma_{-2 m}(\widetilde{D}_0^{-2 m})\big] .
\end{align}

 \begin{lem}
The symbols of  $\mathcal{P}$  and  $\mathcal{Q}$  are given:
\begin{align}
&\sigma_{1}(\mathcal{P})=\sqrt{-1}\widetilde{c}(u)\widetilde{c}(\xi); \nonumber\\
	&\sigma_{1}(\mathcal{Q})=\sqrt{-1}\widetilde{c}(v)\widetilde{c}(\xi);\nonumber\\
&\sigma_{0}(\mathcal{P})=-\frac{1}{4} \sum_{p,s,t=1}^{2m} w_{st}(e_p)\widetilde{c}(u)\widetilde{c}(e_p)c(e_s)c(e_t)+\frac{1}{4} \sum_{p,s,t=1}^{2m} w_{st}(e_p)\widetilde{c}(u)\widetilde{c}(e_p)\widehat{c}(e_s)\widehat{c}(e_t);\nonumber\\
	&\sigma_{0}(\mathcal{Q})=-\frac{1}{4} \sum_{p,s,t=1}^{2m} w_{st}(e_p)\widetilde{c}(v)\widetilde{c}(e_p)c(e_s)c(e_t)+\frac{1}{4} \sum_{p,s,t=1}^{2m} w_{st}(e_p)\widetilde{c}(v)\widetilde{c}(e_p)\widehat{c}(e_s)\widehat{c}(e_t).\nonumber
\end{align}
 \end{lem}

Further, by the composition formula of pseudo-differential operators, we get the following lemma.

 \begin{lem}
	The symbols of $\mathcal{PQ}$ are given:
	\begin{align}
		\sigma_{0}(\mathcal{PQ})=&\sigma_{0}(\mathcal{P}) \sigma_{0}(\mathcal{Q})+(-i) \partial_{\xi_{j}}\left[\sigma_{1}(\mathcal{P})\right] \partial_{x_{j}}\left[\sigma_{0}(\mathcal{Q})\right]+(-i) \partial_{\xi_{j}}\left[\sigma_{0}(\mathcal{P})\right] \partial_{x_{j}}\left[\sigma_{1}(\mathcal{Q})\right]\nonumber\\
		=&\;\;\;\frac{1}{16}\sum_{p,s,t,\hat{p},\hat{s},\hat{t}=1}^{2m} w_{st}(e_p) w_{\hat{s} \hat{t}}(e_{\hat{p}})\widetilde{c}(u)\widetilde{c}(e_p)c(e_s)c(e_t)\widetilde{c}(v)\widetilde{c}(e_{\hat{p}})c(e_{\hat{s}})c(e_{\hat{t}})\nonumber\\
		&-\frac{1}{16}\sum_{p,s,t,\hat{p},\hat{s},\hat{t}=1}^{2m} w_{st}(e_p) w_{\hat{s} \hat{t}}(e_{\hat{p}})\widetilde{c}(u)\widetilde{c}(e_p)c(e_s)c(e_t)\widetilde{c}(v)\widetilde{c}(e_{\hat{p}})\widehat{c}(e_{\hat{s}})\widehat{c}(e_{\hat{t}})\nonumber\\
		&-\frac{1}{16}\sum_{p,s,t,\hat{p},\hat{s},\hat{t}=1}^{2m} w_{st}(e_p) w_{\hat{s} \hat{t}}(e_{\hat{p}})\widetilde{c}(u)\widetilde{c}(e_p)\widehat{c}(e_s)\widehat{c}(e_t)\widetilde{c}(v)\widetilde{c}(e_{\hat{p}})c(e_{\hat{s}})c(e_{\hat{t}})\nonumber\\
		&+\frac{1}{16}\sum_{p,s,t,\hat{p},\hat{s},\hat{t}=1}^{2m} w_{st}(e_p) w_{\hat{s} \hat{t}}(e_{\hat{p}})\widetilde{c}(u)\widetilde{c}(e_p)\widehat{c}(e_s)\widehat{c}(e_t)\widetilde{c}(v)\widetilde{c}(e_{\hat{p}})\widehat{c}(e_{\hat{s}})\widehat{c}(e_{\hat{t}})\nonumber\\
&+\frac{1}{8}\sum_{j,\hat{p},\hat{s},\hat{t}=1}^{2m} {\operatorname{R}}_{j\hat{p}\hat{t}\hat{s}}\widetilde{c}(u)\widetilde{c}(dx_j)\widetilde{c}(v)\widetilde{c}(e_{\hat{p}})c(e_{\hat{s}})c(e_{\hat{t}})\nonumber\\
&-\frac{1}{8}\sum_{j,\hat{p},\hat{s},\hat{t}=1}^{2m} {\operatorname{R}}_{j\hat{p}\hat{t}\hat{s}}\widetilde{c}(u)\widetilde{c}(dx_j)\widetilde{c}(v)\widetilde{c}(e_{\hat{p}})\widehat{c}(e_{\hat{s}})\widehat{c}(e_{\hat{t}}) \nonumber\\
&+\frac{1}{4}\sum_{j,\hat{p},\hat{s},\hat{t},\gamma=1}^{2m}w_{\hat{s} \hat{t}}(e_{\hat{p}})\partial {x_j}(v_\gamma)\widetilde{c}(u)\widetilde{c}(dx_j)\widetilde{c}(e_\gamma)\widetilde{c}(e_{\hat{p}})\widehat{c}(e_{\hat{s}})\widehat{c}(e_{\hat{t}})\nonumber\\
&-\frac{1}{4}\sum_{j,\hat{p},\hat{s},\hat{t},\gamma=1}^{2m}w_{\hat{s} \hat{t}}(e_{\hat{p}})\partial {x_j}(v_\gamma)\widetilde{c}(u)\widetilde{c}(dx_j)\widetilde{c}(e_\gamma)\widetilde{c}(e_{\hat{p}})c(e_{\hat{s}})c(e_{\hat{t}});\nonumber\\
		\sigma_{1}(\mathcal{PQ})=&\sigma_{1}(\mathcal{P}) \sigma_{0}(\mathcal{Q})+\sigma_{0}(\mathcal{P}) \sigma_{1}(\mathcal{Q})+(-i) \partial_{\xi_{j}}\left[\sigma_{1}(\mathcal{P})\right] \partial_{x_{j}}\left[\sigma_{1}(\mathcal{Q})\right] \nonumber\\
		=&-\frac{i}{4}\sum_{p,s,t=1}^{2m} w_{st}(e_p) \widetilde{c}(u)\widetilde{c}(\xi)\widetilde{c}(v)\widetilde{c}(e_p)c(e_s)c(e_t)+\frac{i}{4}\sum_{p,s,t=1}^{2m} w_{st}(e_p) \widetilde{c}(u)\widetilde{c}(\xi)\widetilde{c}(v)\widetilde{c}(e_p)\widehat{c}(e_s)\widehat{c}(e_t)\nonumber\\
		&-\frac{i}{4}\sum_{p,s,t=1}^{2m} w_{st}(e_p) \widetilde{c}(u)\widetilde{c}(e_p)c(e_s)c(e_t)\widetilde{c}(v)\widetilde{c}(\xi)\nonumber+\frac{i}{4}\sum_{p,s,t=1}^{2m} w_{st}(e_p) \widetilde{c}(u)\widetilde{c}(e_p)\widehat{c}(e_s)\widehat{c}(e_t)\widetilde{c}(v)\widetilde{c}(\xi)\nonumber\\
		&+i\sum_{j,\gamma=1}^{2m}\partial x_j(v_\gamma)\widetilde{c}(v)\widetilde{c}(dx_j)\widetilde{c}(e_\gamma)\widetilde{c}(\xi);\nonumber\\
		\sigma_{2}(\mathcal{PQ})=&\sigma_{1}(\mathcal{P})\sigma_{1}(\mathcal{Q})=-\widetilde{c}(u)\widetilde{c}(\xi)\widetilde{c}(v)\widetilde{c}(\xi).\nonumber
	\end{align}
\end{lem}

\noindent {\bf (I-1)} For $\sigma_{0}(\mathcal{PQ}) \sigma_{-2 m}(\widetilde{D}_0^{-2 m})(x_{0})$:
\begin{align}\label{0-2m}
\sigma_{0}(\mathcal{PQ}) \sigma_{-2 m}(\widetilde{D}_0^{-2 m})(x_{0})
=&\frac{1}{8}\|\xi\|^{-2 m}\sum_{j,\hat{p},\hat{s},\hat{t}=1}^{2m}  {\operatorname{R}}_{j \hat{p} \hat{t} \hat{s}}\widetilde{c}(u)\widetilde{c}(dx_j)\widetilde{c}(v)\widetilde{c}(e_{\hat{p}})c(e_{\hat{s}})c(e_{\hat{t}})\nonumber\\
&-\frac{1}{8}\|\xi\|^{-2 m}\sum_{j,\hat{p},\hat{s},\hat{t}=1}^{2m}  {\operatorname{R}}_{j \hat{p} \hat{t} \hat{s}}\widetilde{c}(u)\widetilde{c}(dx_j)\widetilde{c}(v)\widetilde{c}(e_{\hat{p}})\widehat{c}(e_{\hat{s}})\widehat{c}(e_{\hat{t}}).
\end{align}
By calculation, we have the following equalities:
\begin{align}\label{ee1}
&\widetilde{c}(X)c(Y)+c(Y)\widetilde{c}(X)=-(a_0+b_0)g^{TM}(X,Y);\nonumber\\
&\widetilde{c}(X)\widetilde{c}(Y)+\widetilde{c}(Y)\widetilde{c}(X)=-2a_0b_0g^{TM}(X,Y);\nonumber\\
&\widetilde{c}(X)\widehat{c}(Y)+\widehat{c}(Y)\widetilde{c}(X)=(a_0-b_0)g^{TM}(X,Y).
\end{align}

 \noindent{\bf (I-1-$\mathbb{A}$)}By \eqref{ee1}, we have 

\begin{align}\label{t2}
	&{\rm tr}\bigg(\sum_{j,\hat{p},\hat{s},\hat{t}=1}^{2m}\widetilde{c}(u)\widetilde{c}(e_j)\widetilde{c}(v)\widetilde{c}(e_{\hat{p}})c(e_{\hat{s}})c(e_{\hat{t}})\bigg)\nonumber\\
	=&\sum_{j,\hat{p},\hat{s},\hat{t},r,f=1}^{2m}u_rv_f\bigg[a_0^2b_0^2\Big(\delta_r^f\delta_j^{\hat{p}}\delta_{\hat{s}}^{\hat{t}}-\delta_r^j\delta_f^{\hat{p}}\delta_{\hat{s}}^{\hat{t}}-\delta_r^{\hat{p}}\delta_j^f\delta_{\hat{s}}^{\hat{t}}\Big)+\frac{a_0b_0(a_0+b_0)^2}{4}\Big(\delta_r^j\delta_f^{\hat{s}}\delta_{\hat{p}}^{\hat{t}}-\delta_r^j\delta_f^{\hat{t}}\delta_{\hat{p}}^{\hat{s}}-\delta_r^f\delta_j^{\hat{s}}\delta_{\hat{p}}^{\hat{t}}\nonumber\\
&+\delta_r^f\delta_j^{\hat{t}}\delta_{\hat{p}}^{\hat{s}}+\delta_r^{\hat{p}}\delta_j^{\hat{s}}\delta_f^{\hat{t}}-\delta_r^{\hat{p}}\delta_j^{\hat{t}}\delta_f^{\hat{s}}+\delta_r^{\hat{s}}\delta_j^f\delta_{\hat{p}}^{\hat{t}}-\delta_r^{\hat{s}}\delta_j^{\hat{p}}\delta_f^{\hat{t}}+\delta_r^{\hat{s}}\delta_j^{\hat{t}}\delta_f^{\hat{p}}
-\delta_r^{\hat{t}}\delta_j^f\delta_{\hat{p}}^{\hat{s}}+\delta_r^{\hat{t}}\delta_j^{\hat{p}}\delta_f^{\hat{s}}-\delta_r^{\hat{t}}\delta_j^{\hat{s}}\delta_f^{\hat{p}}\Big)\bigg]{\rm tr}[id],
\end{align}
then
\begin{align}
	&\int_{\|\xi\|=1}\operatorname{tr}\biggl\{ \frac{1}{8}\|\xi\|^{-2 m}\sum_{j,\hat{p},\hat{s},\hat{t}=1}^{2m}  {\operatorname{R}}_{j \hat{p} \hat{t} \hat{s}}\widetilde{c}(u)\widetilde{c}(dx_j)\widetilde{c}(v)\widetilde{c}(e_{\hat{p}})c(e_{\hat{s}})c(e_{\hat{t}}) \biggr\}(x_0)\sigma(\xi)\nonumber\\
=&\frac{a_0b_0(a_0+b_0)^2}{4}\bigg(\frac{1}{4}g(u,v)\sum_{j,\hat{p}=1}^{2m}R(e_j,e_{\hat{p}},e_j,e_{\hat{p}})-\frac{1}{2}\sum_{\hat{p}=1}^{2m}R(u,e_{\hat{p}},v,e_{\hat{p}})\bigg){\rm Vol}(S^{n-1}){\rm tr}[id]\nonumber\\
=&\frac{a_0b_0(a_0+b_0)^2}{4}\bigg(\frac{1}{4} s g(u,v)-\frac{1}{2}{\rm Ric}(u,v)\bigg){\rm Vol}(S^{n-1}){\rm tr}[id].
\end{align}

\noindent{\bf (I-1-$\mathbb{B}$)}By \eqref{ee1}, we also have

\begin{align}\label{t2}
	&{\rm tr}\bigg(\sum_{j,\hat{p},\hat{s},\hat{t}=1}^{2m}\widetilde{c}(u)\widetilde{c}(e_j)\widetilde{c}(v)\widetilde{c}(e_{\hat{p}})\widehat{c}(e_{\hat{s}})\widehat{c}(e_{\hat{t}})\bigg)\nonumber\\
	=&\sum_{j,\hat{p},\hat{s},\hat{t},r,f=1}^{2m}u_rv_f\bigg[a_0^2b_0^2\Big(\delta_r^j\delta_f^{\hat{p}}\delta_{\hat{s}}^{\hat{t}}-\delta_r^f\delta_j^{\hat{p}}\delta_{\hat{s}}^{\hat{t}}+\delta_r^{\hat{p}}\delta_j^f\delta_{\hat{s}}^{\hat{t}}\Big)+\frac{a_0b_0(a_0-b_0)^2}{4}\Big(\delta_r^j\delta_f^{\hat{s}}\delta_{\hat{p}}^{\hat{t}}-\delta_r^j\delta_f^{\hat{t}}\delta_{\hat{p}}^{\hat{s}}-\delta_r^f\delta_j^{\hat{s}}\delta_{\hat{p}}^{\hat{t}}\nonumber\\
&+\delta_r^f\delta_j^{\hat{t}}\delta_{\hat{p}}^{\hat{s}}+\delta_r^{\hat{p}}\delta_j^{\hat{s}}\delta_f^{\hat{t}}-\delta_r^{\hat{p}}\delta_j^{\hat{t}}\delta_f^{\hat{s}}+\delta_r^{\hat{s}}\delta_j^f\delta_{\hat{p}}^{\hat{t}}-\delta_r^{\hat{s}}\delta_j^{\hat{p}}\delta_f^{\hat{t}}+\delta_r^{\hat{s}}\delta_j^{\hat{t}}\delta_f^{\hat{p}}
-\delta_r^{\hat{t}}\delta_j^f\delta_{\hat{p}}^{\hat{s}}+\delta_r^{\hat{t}}\delta_j^{\hat{p}}\delta_f^{\hat{s}}-\delta_r^{\hat{t}}\delta_j^{\hat{s}}\delta_f^{\hat{p}}\Big)\bigg]{\rm tr}[id],
\end{align}
then
\begin{align}
	&\int_{\|\xi\|=1}\operatorname{tr}\biggl\{ \frac{1}{8}\|\xi\|^{-2 m}\sum_{j,\hat{p},\hat{s},\hat{t}=1}^{2m}  {\operatorname{R}}_{j \hat{p} \hat{t} \hat{s}}\widetilde{c}(u)\widetilde{c}(dx_j)\widetilde{c}(v)\widetilde{c}(e_{\hat{p}})\widehat{c}(e_{\hat{s}})\widehat{c}(e_{\hat{t}}) \biggr\}(x_0)\sigma(\xi)\nonumber\\
=&\frac{a_0b_0(a_0-b_0)^2}{4}\bigg(\frac{1}{4}g(u,v)\sum_{j,\hat{p}=1}^{2m}R(e_j,e_{\hat{p}},e_j,e_{\hat{p}})-\frac{1}{2}\sum_{\hat{p}=1}^{2m}R(u,e_{\hat{p}},v,e_{\hat{p}})\bigg){\rm Vol}(S^{n-1}){\rm tr}[id]\nonumber\\
=&\frac{a_0b_0(a_0-b_0)^2}{4}\bigg(\frac{1}{4} s g(u,v)-\frac{1}{2}{\rm Ric}(u,v)\bigg){\rm Vol}(S^{n-1}){\rm tr}[id].
\end{align}
Therefore, we get
\begin{align}
	&\int_{\|\xi\|=1} \operatorname{tr}\big[\sigma_{0}(\mathcal{PQ}) \sigma_{-2 m}(\widetilde{D}_0^{-2 m})(x_{0})\big] \sigma(\xi)\nonumber\\
	=&a_0^2b_0^2\bigg(\frac{1}{4} s g(u,v)-\frac{1}{2}{\rm Ric}(u,v)\bigg){\rm Vol}(S^{n-1}){\rm tr}[id].
\end{align}

\noindent{\bf (I-2)} For $\sigma_{1}(\mathcal{PQ}) \sigma_{-2 m-1}(\widetilde{D}_0^{-2 m})(x_{0})$:\\

\noindent Obviously, we have
\begin{align}
	&\sigma_{1}(\mathcal{PQ}) \sigma_{-2 m-1}(\widetilde{D}_0^{-2 m})(x_{0})=0,\nonumber
\end{align}
so
\begin{align}
	&\int_{\|\xi\|=1} \operatorname{tr}\big[\sigma_{1}(\mathcal{PQ}) \sigma_{-2 m-1}(\widetilde{D}_0^{-2 m})(x_{0})\big] \sigma(\xi)=0.\nonumber
\end{align}

\noindent{\bf (I-3)} For $\sigma_{2}(\mathcal{PQ}) \sigma_{-2 m-2}(\widetilde{D}_0^{-2 m})(x_{0})$:
\begin{align}\label{2-2m-2}
&\sigma_{2}(\mathcal{PQ}) \sigma_{-2 m-2}(\widetilde{D}_0^{-2 m})(x_{0})\nonumber\\
=&-\frac{m(m+1)}{3}\|\xi\|^{-2 m-4}\sum_{a,b=1}^{2m} \operatorname{Ric}_{a b} \xi_{a} \xi_{b}\widetilde{c}(u)\widetilde{c}(\xi)\widetilde{c}(v)\widetilde{c}(\xi) \nonumber\\
&+\frac{ m (m+1)}{4}\|\xi\|^{-2 m-4} \sum_{a,b,s,t=1}^{2m}\operatorname{R}_{b a t s} \xi_{a} \xi_{b}\widetilde{c}(u)\widetilde{c}(\xi)\widetilde{c}(v)\widetilde{c}(\xi)c(e_s) c(e_t)\nonumber\\
&-\frac{ m (m+1)}{4}\|\xi\|^{-2 m-4} \sum_{a,b,s,t=1}^{2m}\operatorname{R}_{b a t s} \xi_{a} \xi_{b}\widetilde{c}(u)\widetilde{c}(\xi)\widetilde{c}(v)\widetilde{c}(\xi)\widehat{c}(e_s) \widehat{c}(e_t) \nonumber\\
&+\frac{ m }{8}\|\xi\|^{-2 m-2}\sum_{i,j,k,l=1}^{2m}R_{ijkl}\widetilde{c}(u)\widetilde{c}(\xi)\widetilde{c}(v)\widetilde{c}(\xi)\widehat{c}(e_i) \widehat{c}(e_j)c(e_k)c(e_l)\nonumber\\
&+\frac{m}{4}\|\xi\|^{-2 m-2}s\widetilde{c}(u)\widetilde{c}(\xi)\widetilde{c}(v)\widetilde{c}(\xi).
\end{align}

\noindent{\bf (I-3-$\mathbb{A}$)}Due to
\begin{align}\label{t7}
	{\rm tr}\big[\widetilde{c}(u)\widetilde{c}(\xi)\widetilde{c}(v)\widetilde{c}(\xi)\big]=\sum_{f,g=1}^{2m}\xi_f\xi_g{\rm tr}\big[\widetilde{c}(u)\widetilde{c}(e_f)\widetilde{c}(v)\widetilde{c}(e_g)\big],
\end{align}
then
\begin{align}
	&\int_{\|\xi\|=1} \operatorname{tr}\biggl\{-\frac{m(m+1)}{3}\|\xi\|^{-2 m-4}\sum_{a,b=1}^{2m} \operatorname{Ric}_{a b} \xi_{a} \xi_{b} \widetilde{c}(u)\widetilde{c}(\xi)\widetilde{c}(v)\widetilde{c}(\xi)\biggr\}(x_0)\sigma(\xi)\nonumber\\
=&-\frac{m(m+1)}{3}\times\frac{1}{2m(2m+2)}\sum_{a,b,f,g=1}^{2m}\operatorname{Ric}_{a b}(\delta_a^b\delta_f^g+\delta_a^f\delta_b^g+\delta_a^g\delta_b^f){\rm tr}\big[\widetilde{c}(u)\widetilde{c}(e_f)\widetilde{c}(v)\widetilde{c}(e_g)\big]{\rm Vol}(S^{n-1})\nonumber\\
=&-\frac{1}{12}\sum_{l,a,f=1}^{2m}R(e_l,e_a,e_l,e_a){\rm tr}\big[\widetilde{c}(u)\widetilde{c}(e_f)\widetilde{c}(v)\widetilde{c}(e_f)\big]{\rm Vol}(S^{n-1})\nonumber\\
&-\frac{1}{12}\sum_{l,a,b=1}^{2m}R(e_l,e_a,e_l,e_b){\rm tr}\big[\widetilde{c}(u)\widetilde{c}(e_a)\widetilde{c}(v)\widetilde{c}(e_b)\big]{\rm Vol}(S^{n-1})\nonumber\\
&-\frac{1}{12}\sum_{l,a,b=1}^{2m}R(e_l,e_a,e_l,e_b){\rm tr}\big[\widetilde{c}(u)\widetilde{c}(e_b)\widetilde{c}(v)\widetilde{c}(e_a)\big]{\rm Vol}(S^{n-1})\nonumber\\
=&a_0^2b_0^2\bigg(\frac{m}{6} s g(u,v)-\frac{1}{3}{\rm Ric}(u,v)\bigg){\rm Vol}(S^{n-1}){\rm tr}[id].
\end{align}

\noindent{\bf (I-3-$\mathbb{B}$)}

\begin{align}
	&\int_{\|\xi\|=1}\operatorname{tr} \biggl\{\frac{ m (m+1)}{4}\|\xi\|^{-2 m-4} \sum_{a,b,s,t=1}^{2m}\operatorname{R}_{b a t s} \xi_{a} \xi_{b}\widetilde{c}(u)\widetilde{c}(\xi)\widetilde{c}(v)\widetilde{c}(\xi)c(e_s) c(e_t) \biggr\}(x_0)\sigma(\xi)\nonumber\\
=&\frac{ m (m+1)}{4}\times\frac{1}{2m(2m+2)}\sum_{a,b,f,g,s,t=1}^{2m}\operatorname{R}_{b a t s} (\delta_a^b\delta_f^g+\delta_a^f\delta_b^g+\delta_a^g\delta_b^f) \nonumber\\
&\times{\rm tr}\big[\widetilde{c}(u)\widetilde{c}(e_f)\widetilde{c}(v)\widetilde{c}(e_g)c(e_s) c(e_t)\big]{\rm Vol}(S^{n-1})\nonumber\\
=&\frac{1}{16}\sum_{a,b,s,t=1}^{2m}\operatorname{R}_{b a t s} {\rm tr}\big[\widetilde{c}(u)\widetilde{c}(e_a)\widetilde{c}(v)\widetilde{c}(e_b)c(e_s) c(e_t)\big]{\rm Vol}(S^{n-1}) \nonumber\\
&+\frac{1}{16}\sum_{a,b,s,t=1}^{2m}\operatorname{R}_{b a t s} {\rm tr}\big[\widetilde{c}(u)\widetilde{c}(e_b)\widetilde{c}(v)\widetilde{c}(e_a)c(e_s) c(e_t)\big]{\rm Vol}(S^{n-1}) \nonumber\\
	=&0.\nonumber
\end{align}

\noindent{\bf (I-3-$\mathbb{C}$)}

\begin{align}
	&\int_{\|\xi\|=1}\operatorname{tr} \biggl\{-\frac{ m (m+1)}{4}\|\xi\|^{-2 m-4} \sum_{a,b,s,t=1}^{2m}\operatorname{R}_{b a t s} \xi_{a} \xi_{b}\widetilde{c}(u)\widetilde{c}(\xi)\widetilde{c}(v)\widetilde{c}(\xi)\widehat{c}(e_s) \widehat{c}(e_t) \biggr\}(x_0)\sigma(\xi)\nonumber\\
=&-\frac{ m (m+1)}{4}\times\frac{1}{2m(2m+2)}\sum_{a,b,f,g,s,t=1}^{2m}\operatorname{R}_{b a t s} (\delta_a^b\delta_f^g+\delta_a^f\delta_b^g+\delta_a^g\delta_b^f)\nonumber\\
&\times {\rm tr}\big[\widetilde{c}(u)\widetilde{c}(e_f)\widetilde{c}(v)\widetilde{c}(e_g)\widehat{c}(e_s) \widehat{c}(e_t)\big] {\rm Vol}(S^{n-1})\nonumber\\
=&-\frac{1}{16}\sum_{a,b,s,t=1}^{2m}\operatorname{R}_{b a t s} {\rm tr}\big[\widetilde{c}(u)\widetilde{c}(e_a)\widetilde{c}(v)\widetilde{c}(e_b)\widehat{c}(e_s) \widehat{c}(e_t)\big] {\rm Vol}(S^{n-1}) \nonumber\\
&-\frac{1}{16}\sum_{a,b,s,t=1}^{2m}\operatorname{R}_{b a t s} {\rm tr}\big[\widetilde{c}(u)\widetilde{c}(e_b)\widetilde{c}(v)\widetilde{c}(e_a)\widehat{c}(e_s) \widehat{c}(e_t)\big] {\rm Vol}(S^{n-1}) \nonumber\\
	=&0.\nonumber
\end{align}

\noindent{\bf (I-3-$\mathbb{D}$)}

\begin{align}\label{ee2}
	&\int_{\|\xi\|=1}\operatorname{tr} \biggl\{\frac{ m }{8}\|\xi\|^{-2 m-2}\sum_{i,j,k,l=1}^{2m}R_{ijkl}\widetilde{c}(u)\widetilde{c}(\xi)\widetilde{c}(v)\widetilde{c}(\xi)\widehat{c}(e_i) \widehat{c}(e_j)c(e_k)c(e_l)\biggr\}(x_0)\sigma(\xi)\nonumber\\
	=&\int_{\|\xi\|=1}\operatorname{tr} \biggl\{\frac{ m }{8}\|\xi\|^{-2 m-2}\sum_{i,j,k,l=1}^{2m}R_{ijkl}\xi_{f} \xi_{g}\widetilde{c}(u)\widetilde{c}(e_f)\widetilde{c}(v)\widetilde{c}(e_g)\widehat{c}(e_i) \widehat{c}(e_j)c(e_k)c(e_l)\biggr\}(x_0)\sigma(\xi)\nonumber\\
	=&\frac{ 1}{16}\sum_{i,j,k,l=1}^{2m}\operatorname{R}_{ijkl} {\rm tr}\big[\widetilde{c}(u)\widetilde{c}(e_f)\widetilde{c}(v)\widetilde{c}(e_f)\widehat{c}(e_i) \widehat{c}(e_j)c(e_k)c(e_l)\big]{\rm Vol}(S^{n-1}).
\end{align}
By computations, we have
\begin{align}\label{ee3}
&{\rm tr}\big[\widetilde{c}(u)\widetilde{c}(e_f)\widetilde{c}(v)\widetilde{c}(e_f)\widehat{c}(e_i) \widehat{c}(e_j)c(e_k)c(e_l)\big]\nonumber\\
=&(m-2)a_0b_0\times\sum_{i,j,k,l,r,p=1}^{2m}u_rv_p\bigg[a_0b_0\delta_k^l\delta_r^p\delta_i^j+\frac{(a_0-b_0)^2}{4}\Big(\delta_k^l\delta_r^j\delta_p^i-\delta_k^l\delta_r^i\delta_p^j\Big)\nonumber\\
&+\frac{(a_0+b_0)^2}{4}\Big(\delta_k^p\delta_l^r\delta_i^j-\delta_k^r\delta_l^p\delta_i^j\Big)\bigg]{\rm tr}[id].
\end{align}
Substituting \eqref{ee3} into \eqref{ee2}, we get
\begin{align}
&\int_{\|\xi\|=1}\operatorname{tr} \biggl\{\frac{ m }{8}\|\xi\|^{-2 m-2}\sum_{i,j,k,l=1}^{2m}R_{ijkl}\widetilde{c}(u)\widetilde{c}(\xi)\widetilde{c}(v)\widetilde{c}(\xi)\widehat{c}(e_i) \widehat{c}(e_j)c(e_k)c(e_l)\biggr\}(x_0)\sigma(\xi)=0.\nonumber
\end{align}

\noindent{\bf (I-3-$\mathbb{E}$)}

\begin{align}
	&\int_{\|\xi\|=1}\operatorname{tr} \biggl\{\frac{m}{4}\|\xi\|^{-2 m-2}s\widetilde{c}(u)\widetilde{c}(\xi)\widetilde{c}(v)\widetilde{c}(\xi)\biggr\}(x_0)\sigma(\xi)\nonumber\\
	=&\int_{\|\xi\|=1}\operatorname{tr} \biggl\{\frac{m}{4}\|\xi\|^{-2 m-2}s\xi_{f} \xi_{g}\widetilde{c}(u)\widetilde{c}(e_f)\widetilde{c}(v)\widetilde{c}(e_g)\biggr\}(x_0)\sigma(\xi)\nonumber\\
	=&\frac{ 1}{8}s{\rm tr}\big[\widetilde{c}(u)\widetilde{c}(e_f)\widetilde{c}(v)\widetilde{c}(e_f)\big]{\rm Vol}(S^{n-1}) \nonumber\\
	=&a_0^2b_0^2\frac{1-m}{4}sg(u,v){\rm Vol}(S^{n-1}){\rm tr}[id].
\end{align}
As a result, we obtain
\begin{align}
	&\int_{\|\xi\|=1} \operatorname{tr}\big[\sigma_{2}(\mathcal{PQ}) \sigma_{-2 m-2}(\widetilde{D}_0^{-2 m})(x_{0})\big] \sigma(\xi)\nonumber\\
	=&a_0^2b_0^2\bigg(\frac{3-m}{12} s g(u,v)-\frac{1}{3}{\rm Ric}(u,v)\bigg){\rm Vol}(S^{n-1}){\rm tr}[id].
\end{align}

\noindent{\bf (I-4)} For $-i\sum_{j=1}^{2m} \partial_{\xi_{j}}\big[\sigma_{2}(\mathcal{P} \mathcal{Q})\big] \partial_{x_{j}}\big[\sigma_{-2 m-1}(\widetilde{D}_0^{-2 m})\big](x_{0})$:

\begin{align}\label{2-2m-1}
	&-i\sum_{j=1}^{2m} \partial_{\xi_{j}}[\sigma_{2}(\mathcal{P} \mathcal{Q})] \partial_{x_{j}}[\sigma_{-2 m-1}(\widetilde{D}_0^{-2 m})](x_{0})\nonumber\\
	=&\;\;\frac{2 m }{3}\|\xi\|^{-2 m-2}\sum_{a,b,j=1}^{2m}  \operatorname{Ric}_{a b}  \xi_{a}\delta^{b}_{j} \bigg(\widetilde{c}(u)\widetilde{c}(dx_j)\widetilde{c}(v)\widetilde{c}(\xi)+\widetilde{c}(u)\widetilde{c}(\xi)\widetilde{c}(v)\widetilde{c}(dx_j)\bigg)\nonumber\\
	&-\frac{ m }{4}\|\xi\|^{-2 m-2} \sum_{a,b,j,s,t=1}^{2m} \operatorname{R}_{b a t s}\xi_{a}\delta^{b}_{j}\bigg(\widetilde{c}(u)\widetilde{c}(dx_j)\widetilde{c}(v)\widetilde{c}(\xi)+\widetilde{c}(u)\widetilde{c}(\xi)\widetilde{c}(v)\widetilde{c}(dx_j) \bigg) c(e_s) c(e_t) \nonumber\\
	&+\frac{ m }{4}\|\xi\|^{-2 m-2} \sum_{a,b,j,s,t=1}^{2m} \operatorname{R}_{b a t s}\xi_{a}\delta^{b}_{j}\bigg(\widetilde{c}(u)\widetilde{c}(dx_j)\widetilde{c}(v)\widetilde{c}(\xi)+\widetilde{c}(u)\widetilde{c}(\xi)\widetilde{c}(v)\widetilde{c}(dx_j) \bigg)\widehat{c}(e_s) \widehat{c}(e_t).
\end{align}

\noindent{\bf (I-4-$\mathbb{A}$)}

\begin{align}
	&\int_{\|\xi\|=1} {\rm tr}\biggl\{\frac{2 m }{3}\|\xi\|^{-2 m-2}\sum_{a,b,j=1}^{2m}  \operatorname{Ric}_{a b}  \xi_{a}\delta^{b}_{j}\bigg(\widetilde{c}(u)\widetilde{c}(dx_j)\widetilde{c}(v)\widetilde{c}(\xi)+\widetilde{c}(u)\widetilde{c}(\xi)\widetilde{c}(v)\widetilde{c}(dx_j)\bigg)\biggr\}(x_0)\sigma(\xi)\nonumber\\
	=&\int_{\|\xi\|=1}{\rm tr} \biggl\{\frac{2 m }{3}\|\xi\|^{-2 m-2}\sum_{a,b,g=1}^{2m}  \operatorname{Ric}_{a b}  \xi_{a}\xi_{g}\bigg(\widetilde{c}(u)\widetilde{c}(e_b)\widetilde{c}(v)\widetilde{c}(e_g)+\widetilde{c}(u)\widetilde{c}(e_g)\widetilde{c}(v)\widetilde{c}(e_b)\bigg)\biggr\}(x_0)\sigma(\xi)\nonumber\\
	=&\frac{1}{3}\sum_{a,b=1}^{2m}  \operatorname{Ric}_{a b}  {\rm tr}\big[\widetilde{c}(u)\widetilde{c}(e_b)\widetilde{c}(v)\widetilde{c}(e_a)+\widetilde{c}(u)\widetilde{c}(e_a)\widetilde{c}(v)\widetilde{c}(e_b)\big]{\rm Vol}(S^{n-1})\nonumber\\
	=&\frac{2}{3}\sum_{a,b=1}^{2m}  \operatorname{Ric}_{a b}  {\rm tr}\big[\widetilde{c}(u)\widetilde{c}(e_b)\widetilde{c}(v)\widetilde{c}(e_a)\big]{\rm Vol}(S^{n-1})\nonumber\\
	=&a_0^2b_0^2\bigg(\frac{4}{3}\operatorname{Ric}(u,v)-\frac{2}{3}s g(u,v)\bigg) {\rm Vol}(S^{n-1}){\rm tr}[id].
\end{align}

\noindent{\bf (I-4-$\mathbb{B}$)}

\begin{align}
	&\int_{\|\xi\|=1}\operatorname{tr} \biggl\{-\frac{ m }{4}\|\xi\|^{-2 m-2} \sum_{a,b,j,s,t=1}^{2m} \operatorname{R}_{b a t s}\xi_{a}\delta^{b}_{j}\bigg(\widetilde{c}(u)\widetilde{c}(dx_j)\widetilde{c}(v)\widetilde{c}(\xi) \nonumber\\
&+\widetilde{c}(u)\widetilde{c}(\xi)\widetilde{c}(v)\widetilde{c}(dx_j)\bigg) c(e_s) c(e_t)\biggr\}(x_0)\sigma(\xi)\nonumber\\
	=&\int_{\|\xi\|=1}\operatorname{tr} \biggl\{-\frac{ m }{4}\|\xi\|^{-2 m-2} \sum_{a,b,s,t,g=1}^{2m} \operatorname{R}_{b a t s}\xi_{a}\xi_{g}\bigg(\widetilde{c}(u)\widetilde{c}(e_b)\widetilde{c}(v)\widetilde{c}(e_g) \nonumber\\
& +\widetilde{c}(u)\widetilde{c}(e_g)\widetilde{c}(v)\widetilde{c}(e_b) \bigg)c(e_s) c(e_t)\biggr\}(x_0)\sigma(\xi)\nonumber\\
=&-\frac{1 }{8}\sum_{a,b,s,t=1}^{2m} \operatorname{R}_{b a t s}\operatorname{tr}\big[\widetilde{c}(u)\widetilde{c}(e_b)\widetilde{c}(v)\widetilde{c}(e_a) c(e_s) c(e_t)\big]{\rm Vol}(S^{n-1})\nonumber\\
&-\frac{1 }{8}\sum_{a,b,s,t=1}^{2m} \operatorname{R}_{ab t s}\operatorname{tr}\big[\widetilde{c}(u)\widetilde{c}(e_b)\widetilde{c}(v)\widetilde{c}(e_a)c(e_s) c(e_t)\big]{\rm Vol}(S^{n-1})\nonumber\\
	=&0.\nonumber
\end{align}

\noindent{\bf (I-4-$\mathbb{C}$)}

\begin{align}
	&\int_{\|\xi\|=1}\operatorname{tr} \biggl\{-\frac{ m }{4}\|\xi\|^{-2 m-2} \sum_{a,b,j,s,t=1}^{2m} \operatorname{R}_{b a t s}\xi_{a}\delta^{b}_{j}\bigg(\widetilde{c}(u)\widetilde{c}(dx_j)\widetilde{c}(v)\widetilde{c}(\xi) \nonumber\\ 
&+\widetilde{c}(u)\widetilde{c}(\xi)\widetilde{c}(v)\widetilde{c}(dx_j)\bigg) \widehat{c}(e_s)\widehat{ c}(e_t)\biggr\}(x_0)\sigma(\xi)\nonumber\\
	=&\int_{\|\xi\|=1}\operatorname{tr} \biggl\{-\frac{ m }{4}\|\xi\|^{-2 m-2} \sum_{a,b,s,t,g=1}^{2m} \operatorname{R}_{b a t s}\xi_{a}\xi_{g}\bigg(\widetilde{c}(u)\widetilde{c}(e_b)\widetilde{c}(v)\widetilde{c}(e_g) \nonumber\\
&+\widetilde{c}(u)\widetilde{c}(e_g)\widetilde{c}(v)\widetilde{c}(e_b) \bigg)\widehat{c}(e_s)\widehat{ c}(e_t)\biggr\}(x_0)\sigma(\xi)\nonumber\\
=&-\frac{1 }{8}\sum_{a,b,s,t=1}^{2m} \operatorname{R}_{b a t s}\operatorname{tr}\big[\widetilde{c}(u)\widetilde{c}(e_b)\widetilde{c}(v)\widetilde{c}(e_a)\widehat{c}(e_s)\widehat{ c}(e_t)\big] {\rm Vol}(S^{n-1})\nonumber\\
&-\frac{1 }{8}\sum_{a,b,s,t=1}^{2m} \operatorname{R}_{a b t s}\operatorname{tr}\big[\widetilde{c}(u)\widetilde{c}(e_b)\widetilde{c}(v)\widetilde{c}(e_a)\widehat{c}(e_s)\widehat{ c}(e_t)\big] {\rm Vol}(S^{n-1})\nonumber\\
	=&0.\nonumber
\end{align}
Thus, we get
\begin{align}
	&\int_{\|\xi\|=1} \operatorname{tr}\bigg[-i \sum_{j=1}^{2m} \partial_{\xi_{j}}\big[\sigma_{2}(\mathcal{P} \mathcal{Q})\big] \partial_{x_{j}}\big[\sigma_{-2 m-1}(\widetilde{D}_0^{-2 m})\big]\bigg] (x_0)\sigma(\xi)\nonumber\\
	=&a_0^2b_0^2\bigg(\frac{4}{3}\operatorname{Ric}(u,v)-\frac{2}{3}s g(u,v)\bigg) {\rm Vol}(S^{n-1}){\rm tr}[id].
\end{align}

\noindent{\bf (I-5)} For $-i \sum_{j=1}^{2m} \partial_{\xi_{j}}\big[\sigma_{1}(\mathcal{P} \mathcal{Q})\big] \partial_{x_{j}}\big[\sigma_{-2 m}(\widetilde{D}_0^{-2 m})\big](x_{0})$:
\begin{align}\label{1-2m}
	\partial_{x_{j}}\big[\sigma_{-2 m}(\widetilde{D}_0^{-2 m})\big](x_0)
	=&\partial_{x_{j}}\bigg[\|\xi\|^{-2 m-2}\sum_{a,b,l,k=1}^{2m} \left(\delta_{a b}-\frac{m}{3} R_{a l b k} x^{l} x^{k}\right) \xi_{a} \xi_{b}\bigg](x_0)\nonumber\\
	=&0,\nonumber
\end{align}
so
\begin{align}
	&\int_{\|\xi\|=1} \operatorname{tr}\bigg[-i\sum_{j=1}^{2m} \partial_{\xi_{j}}\big[\sigma_{1}(\mathcal{P} \mathcal{Q})\big] \partial_{x_{j}}\big[\sigma_{-2 m}(\widetilde{D}_0^{-2 m})\big]\bigg] (x_0)\sigma(\xi)=0.\nonumber
\end{align}

\noindent{\bf (I-6)} For $-\frac{1}{2} \sum_{j ,l=1}^{2m} \partial_{\xi_{j}} \partial_{\xi_{l}}\big[\sigma_{2}(\mathcal{P} \mathcal{Q})\big] \partial_{x_{j}} \partial_{x_{l}}\big[\sigma_{-2 m}(\widetilde{D}_0^{-2 m})\big](x_{0})$:

\begin{align}
\sum_{j, l=1} ^{2m}\partial_{\xi_{j}} \partial_{\xi_{l}}\big[\sigma_{2}(\mathcal{P} \mathcal{Q})\big](x_{0})
=&\sum_{j, l=1} ^{2m}\partial_{\xi_{j}} \partial_{\xi_{l}}\big[-\widetilde{c}(u)\widetilde{c}(\xi)\widetilde{c}(v)\widetilde{c}(\xi)\big]\nonumber\\
=&\sum_{j, l=1} ^{2m}\big[-\widetilde{c}(u)\widetilde{c}(dx_l)\widetilde{c}(v)\widetilde{c}(dx_j)-\widetilde{c}(u)\widetilde{c}(dx_j)\widetilde{c}(v)\widetilde{c}(dx_l)\big],
\end{align}
and
\begin{align}
\sum_{j, l=1} ^{2m}\partial_{x_{j}} \partial_{x_{l}}\big[\sigma_{-2 m}(\widetilde{D}_0^{-2 m})\big](x_{0})
&=\sum_{j, l=1} ^{2m}\partial_{x_{j}} \partial_{x_{l}}\bigg[\|\xi\|^{-2 m-2}\sum_{a,b,\hat{j},k=1}^{2m} \left(\delta_{a b}-\frac{m}{3} R_{a \hat{j} b k} x^{\hat{j}} x^{k}\right) \xi_{a} \xi_{b}\bigg]\nonumber\\
&=-\frac{m}{3}\|\xi\|^{-2 m-2}\sum_{j, l,a,b,\hat{j},k=1} ^{2m}\bigg(R_{a\hat{j}bk}\delta_l^{\hat{j}}\delta_j^k-R_{a\hat{j}bk}\delta_l^k\delta_j^{\hat{j}}\bigg)\xi_a\xi_b.
\end{align}
So
\begin{align}\label{1-2m}
	&-\frac{1}{2} \sum_{j, l=1} ^{2m}\partial_{\xi_{j}} \partial_{\xi_{l}}\big[\sigma_{2}(\mathcal{P} \mathcal{Q})\big] \partial_{x_{j}} \partial_{x_{l}}\big[\sigma_{-2 m}(\widetilde{D}_0^{-2 m})\big](x_{0})\nonumber\\
	=&-\frac{m}{6}\|\xi\|^{-2 m-2}\sum_{j, l,a,b,\hat{j},k=1} ^{2m}\|\xi\|^{-2 m-2}\bigg(R_{a\hat{j}bk}\delta_l^{\hat{j}}\delta_j^k+R_{a\hat{j}bk}\delta_l^k\delta_j^{\hat{j}}\bigg)\xi_a\xi_b\widetilde{c}(u)\widetilde{c}(dx_l)\widetilde{c}(v)\widetilde{c}(dx_j)\nonumber\\
&-\frac{m}{6}\|\xi\|^{-2 m-2}\sum_{j, l,a,b,\hat{j},k=1} ^{2m}\|\xi\|^{-2 m-2}\bigg(R_{a\hat{j}bk}\delta_l^{\hat{j}}\delta_j^k+R_{a\hat{j}bk}\delta_l^k\delta_j^{\hat{j}}\bigg)\xi_a\xi_b\widetilde{c}(u)\widetilde{c}(dx_j)\widetilde{c}(v)\widetilde{c}(dx_l),
\end{align}
then
\begin{align}
	&\int_{\|\xi\|=1} {\rm tr}\biggl\{-\frac{1}{2} \sum_{j, l=1} ^{2m}\partial_{\xi_{j}} \partial_{\xi_{l}}\big[\sigma_{2}(\mathcal{P} \mathcal{Q})\big] \partial_{x_{j}} \partial_{x_{l}}\big[\sigma_{-2 m}(\widetilde{D}_0^{-2 m})\big]\biggr\}(x_0)\sigma(\xi)\nonumber\\
	=&\int_{\|\xi\|=1} \biggl\{-\frac{m}{6}\|\xi\|^{-2 m-2}\sum_{a,b,j,l=1}^{2m}\bigg({\rm R}_{albj}+{\rm R}_{ajbl} \bigg)\xi_{a}\xi_{b}{\rm tr}\big[\widetilde{c}(u)\widetilde{c}(dx_l)\widetilde{c}(v)\widetilde{c}(dx_j)\big]\biggr\}(x_0)\sigma(\xi)\nonumber\\
&+\int_{\|\xi\|=1} \biggl\{-\frac{m}{6}\|\xi\|^{-2 m-2}\sum_{a,b,j,l=1}^{2m}\bigg({\rm R}_{albj}+{\rm R}_{ajbl} \bigg)\xi_{a}\xi_{b}{\rm tr}\big[\widetilde{c}(u)\widetilde{c}(dx_j)\widetilde{c}(v)\widetilde{c}(dx_l)\big]\biggr\}(x_0)\sigma(\xi)\nonumber\\
	=&-\frac{1}{6}\sum_{a,j,l=1}^{2m}{\rm R}_{alaj}{\rm tr}\big[\widetilde{c}(u)\widetilde{c}(e_l)\widetilde{c}(v)\widetilde{c}(e_j)+\widetilde{c}(u)\widetilde{c}(e_j)\widetilde{c}(v)\widetilde{c}(e_l)\big]{\rm Vol}(S^{n-1})\nonumber\\
	=&a_0^2b_0^2\bigg(\frac{1}{3}s g(u,v)-\frac{2}{3}{\rm Ric}(u,v)\bigg) {\rm Vol}(S^{n-1}){\rm tr}[id].
\end{align}
In summary, we get
\begin{align}\label{zabdt}
	&\int_{\|\xi\|=1} {\rm tr}\biggl\{	\sigma_{-2 m}(\mathcal{P} \mathcal{Q} \widetilde{D}_0^{-2 m})\biggr\}(x_0)\sigma(\xi)\nonumber\\
	=&a_0^2b_0^2\bigg(\frac{2-m}{12} s g(u,v)-\frac{1}{6}{\rm Ric}(u,v)\bigg) {\rm Vol}(S^{n-1}){\rm tr}[id].
\end{align}
Since ${\rm tr}[id]=2^{2m}$ and ${\rm Vol}(S^{n-1})=\frac{2 \pi^{m}}{\Gamma\left(m\right)}$, we obtain
\begin{align}\label{z2}
	\mathscr{N}_{1}=&(a_0b_0)^{-m}\mathrm{Wres}\bigg(\widetilde{c}(u)\widetilde{D}\widetilde{c}(v)\widetilde{D}\widetilde{D}_0^{-2m}\bigg)\nonumber\\
	=&(a_0b_0)^{-m+2}2^{2m} \frac{2 \pi^{m}}{\Gamma\left(m\right)}\int_{M}\bigg(\frac{2-m}{12} s g(u,v)-\frac{1}{6}{\rm Ric}(u,v)\bigg)d{\rm Vol}_M.
\end{align}

{\bf Part II)} $\mathscr{N}_{2}=(a_0b_0)^{-m+1}\mathrm{Wres}\bigg(\widetilde{c}(u)\widetilde{c}(v)\widetilde{D}_0^{-2m+2}\bigg)$.

By \eqref{wers}, we need to compute  $\int_{\wedge^*T^*M} \operatorname{tr}\left[\sigma_{-2 m}\left(\widetilde{c}(u)\widetilde{c}(v) \widetilde{D}_0^{-2 m+2}\right)\right](x, \xi) $. Based on the algorithm yielding the principal symbol of a product of pseudo-differential operators in terms of the principal symbols of the factors,  by Lemma \ref{lem2}, we have
\begin{align}
	\sigma_{-2m}(\widetilde{D}_0^{-2 m+2})=& \frac{m(m-1)}{3}\|\xi\|^{-2 m-2}\sum_{a,b=1}^{2m} \operatorname{Ric}_{a b} \xi_{a} \xi_{b}\nonumber\\
		&-\frac{ m (m-1)}{4}\|\xi\|^{-2 m-2} \sum_{a,b,s,t=1}^{2m}\operatorname{R}_{b a t s} c(e_s) c(e_t) \xi_{a} \xi_{b}\nonumber\\
&+\frac{ m (m-1)}{4}\|\xi\|^{-2 m-2} \sum_{a,b,s,t=1}^{2m}\operatorname{R}_{b a t s} \hat{c}(e_s) \hat{c}(e_t) \xi_{a} \xi_{b}\nonumber\\
		&-\frac{1}{8}(m-1)\|\xi\|^{-2 m}\sum_{i,j,k,l=1}^{2m}R_{ijkl}\widehat{c}(e_i) \widehat{c}(e_j)c(e_k) c(e_l)\nonumber\\
		&-\frac{1}{4}(m-1) \|\xi\|^{-2 m}s+O\left(\mathbf{x}\right).
\end{align}
Obviously, the general dimensional symbols of the $\widetilde{c}(u)\widetilde{c}(v)\widetilde{D}_0^{-2 m+2}$ are given:
\begin{align}\label{PD}
	\sigma_{-2 m}\left(\widetilde{c}(u)\widetilde{c}(v)\widetilde{D}_0^{-2 m+2}\right) =& \frac{m(m-1)}{3}\|\xi\|^{-2 m-2}\sum_{a,b=1}^{2m} \operatorname{Ric}_{a b} \xi_{a} \xi_{b}\widetilde{c}(u)\widetilde{c}(v) \nonumber\\
		&-\frac{ m (m-1)}{4}\|\xi\|^{-2 m-2} \sum_{a,b,s,t=1}^{2m}\operatorname{R}_{b a t s}\xi_{a} \xi_{b} \widetilde{c}(u)\widetilde{c}(v)c(e_s) c(e_t) \nonumber\\
&+\frac{ m (m-1)}{4}\|\xi\|^{-2 m-2} \sum_{a,b,s,t=1}^{2m}\operatorname{R}_{b a t s}\xi_{a} \xi_{b} \widetilde{c}(u)\widetilde{c}(v)\widehat{c}(e_s) \widehat{c}(e_t) \nonumber\\
		&-\frac{1}{8}(m-1)\|\xi\|^{-2 m}\sum_{i,j,k,l=1}^{2m}R_{ijkl}\widetilde{c}(u)\widetilde{c}(v)\widehat{c}(e_i) \widehat{c}(e_j)c(e_k) c(e_l)\nonumber\\
		&-\frac{1}{4}(m-1) \|\xi\|^{-2 m}s\widetilde{c}(u)\widetilde{c}(v)+O\left(\mathbf{x}\right).
\end{align}
Next, we integrate each of the above items respectively.
\begin{align}
	&{\bf(II-1)}~\int_{\|\xi\|=1}\operatorname{tr} \biggl\{\frac{m(m-1)}{3}\|\xi\|^{-2 m-2}\sum_{a,b=1}^{2m} \operatorname{Ric}_{a b} \xi_{a} \xi_{b}\widetilde{c}(u)\widetilde{c}(v)\biggr\}(x_0)\sigma(\xi)\nonumber\\
	&\;\;\;\;\;\;\;\;\;\;=-a_0b_0\frac{m-1}{6}s g(u,v){\rm Vol}(S^{n-1}){\rm tr}[id];\nonumber\\
	&{\bf(II-2)}\int_{\|\xi\|=1} {\rm tr}\biggl\{-\frac{ m (m-1)}{4}\|\xi\|^{-2 m-2} \sum_{a,b,s,t=1}^{2m}\operatorname{R}_{b a t s}\xi_{a} \xi_{b}\widetilde{c}(u)\widetilde{c}(v) c(e_s) c(e_t) \biggr\}(x_0)\sigma(\xi)\nonumber\\
	&\;\;\;\;\;\;\;\;\;\;=-\frac{ m-1}{8}\sum_{a,s,t=1}^{2m}\operatorname{R}_{a a t s}{\rm tr}\big[\widetilde{c}(u)\widetilde{c}(v) c(e_s) c(e_t)\big]{\rm Vol}(S^{n-1})  \nonumber\\
	&\;\;\;\;\;\;\;\;\;\;=0;\nonumber\\
	&{\bf(II-3)}\int_{\|\xi\|=1} {\rm tr}\biggl\{\frac{ m (m-1)}{4}\|\xi\|^{-2 m-2} \sum_{a,b,s,t=1}^{2m}\operatorname{R}_{b a t s}\xi_{a} \xi_{b}\widetilde{c}(u)\widetilde{c}(v) \widehat{c}(e_s) \widehat{c}(e_t) \biggr\}(x_0)\sigma(\xi)\nonumber\\
	&\;\;\;\;\;\;\;\;\;\;=0;\nonumber\\
	&{\bf(II-4)}\int_{\|\xi\|=1} {\rm tr}\biggl\{-\frac{1}{8}(m-1)\|\xi\|^{-2 m}\sum_{i,j,k,l=1}^{2m}R_{ijkl}\widetilde{c}(u)\widetilde{c}(v)\widehat{c}(e_i) \widehat{c}(e_j)c(e_k) c(e_l)\biggr\}(x_0)\sigma(\xi)\nonumber\\
	&\;\;\;\;\;\;\;\;\;\;=-\frac{1}{8}(m-1)\sum_{i,j,k,l=1}^{2m}R_{ijkl}{\rm tr}\big[\widetilde{c}(u)\widetilde{c}(v)\widehat{c}(e_i) \widehat{c}(e_j)c(e_k) c(e_l)\big]{\rm Vol}(S^{n-1})\nonumber\\
&\;\;\;\;\;\;\;\;\;\;=0;\nonumber\\
	&{\bf(II-5)}\int_{\|\xi\|=1} {\rm tr}\biggl\{-\frac{1}{4}(m-1) \|\xi\|^{-2 m}s\widetilde{c}(u)\widetilde{c}(v)\biggr\}(x_0)\sigma(\xi)\nonumber\\
	&\;\;\;\;\;\;\;\;\;\;=-\frac{1}{4}(m-1)s{\rm tr}\big[\widetilde{c}(u)\widetilde{c}(v)\big]{\rm Vol}(S^{n-1})\nonumber\\
&\;\;\;\;\;\;\;\;\;\;=\frac{a_0b_0}{4}(m-1)sg(u,v){\rm Vol}(S^{n-1}){\rm tr}[id].\nonumber
\end{align}
In summary, we get
\begin{align}\label{zpdt}
	&\int_{\|\xi\|=1} {\rm tr}\biggl\{\sigma_{-2 m}\left(\widetilde{c}(u)\widetilde{c}(v)  \widetilde{D}_0^{-2 m+2}\right)\biggr\}(x_0)\sigma(\xi)\nonumber\\
	=&a_0b_0\frac{m-1}{12}s g(u,v){\rm Vol}(S^{n-1}){\rm tr}[id].
\end{align}
Further, we obtain
\begin{align}\label{z1}
	\mathscr{N}_{2}=&(a_0b_0)^{-m+1}\mathrm{Wres}\bigg(\widetilde{c}(u)\widetilde{c}(v)\widetilde{D}_0^{-2m+2}\bigg)\nonumber\\
	=&(a_0b_0)^{-m+2}2^{2m} \frac{2 \pi^{m}}{\Gamma\left(m\right)}\int_{M}\frac{m-1}{12}s g(u,v)d{\rm Vol}_M.
\end{align}
Hence, by summing $\mathscr{N}_{1}$ and $\mathscr{N}_{2}$, Theorem \ref{thm} holds.
\end{proof}

\section{Examples}
\label{section:4}
Let $(\mathcal{A},\mathcal{H},D)$ be an $n$-summable unital non-self-adjoint spectral triple, where $\mathcal{A}$ is a noncommutative algebra with involution, acting in the Hilbert space $\mathcal{H}$ while $D$ is a Dirac operator, which maybe non-self-adjoint operator and such that
$[D,a]$ is bounded $\forall a \in \mathcal{A}$. We assume that there exists a generalised algebra of
pseudo-differential operators, which contains $\mathcal{A},D,$ $D^l$ for $l\in \mathbb{Z}$ and there exists a tracial state $\mathcal{W}$  on it, called a noncommutative residue. Moreover, we assume that the noncommutative residue identically vanishes on $TD^{-k}$ for any $k>2m$ and a zero-order operator $T$. Similar to \cite{DL}, we can define a spectral metric functional and a spectral Einstein functional by similar formulas.

\begin{pro}
$\Big(C^{\infty}(M),a_0d+b_0\delta,L^2(M,\wedge^*T^*M)\Big)$ is a non-self-adjoint elliptic spectral triple. When $a_0b_0>0$, it is $\theta$-summable.
\end{pro}
\begin{ex}{\bf $d$-dimensional noncommutative torus}\\
\indent $\mathcal{A}\big(T_\theta^d\big)$ is the algebra of smooth functions on the $d$-dimensional noncommutative torus. $h\in \mathcal{A}\big(T_\theta^d\big)$ is positive element with a bounded inverse. $V$ is a $d$-dimensional complex vector space. $\mathcal{A}\big(T_\theta^d\big)\otimes\wedge^*V$, where $\delta_1,...,\delta_d$ are derivations on $\mathcal{A}\big(T_\theta^d\big)$, $e_1,...,e_d$ are the orthonormal basis of $V$. 

Take $\widetilde{D}=h\big(\Sigma_{j=1}^d\widetilde{c}(e_j)\delta_j\big)h$, which is a non-self-adjoint elliptic operator, then $\Big(\mathcal{A}\big(T_\theta^d\big),\widetilde{D},\mathcal{A}\big(T_\theta^d\big)\otimes\wedge^*V\Big)$ is a noncommutative non-self-adjoint elliptic spectral triple. When $a_0b_0>0$, the spectral triple is $\theta$-summable. So it's interesting to calculate its spectral Einstein functional.
\end{ex}

\begin{ex}
$\widetilde{D}=D+c(V)$, $D$ is Dirac operator. $\Big(C^{\infty}(M),D+c(V),L^2\big(M,S(TM)\big)\Big)$, $M$ is a closed spin manifold, $D$ is Dirac operator, $c(V)$ is a Clifford action. Then it's a non-self-adjoint spectral triple which is $\theta$-summable.
\end{ex}

\begin{ex}
$(\mathcal{A},\mathcal{H},D)$ is a spectral triple. $c(V)=a_0[D,a_1]$, where $a_0, a_1\in\mathcal{A}$, $c(V)$ is not self-adjoint, then $\big(\mathcal{A},\mathcal{H},D+c(V)\big)$ is a non-self-adjoint spectral triple.
\end{ex}

\begin{ex}{\bf Almost commutative $M\times \mathbb{Z}_2$}\\
\indent We assume that $M$ is a closed spin manifold and $\big(C^\infty(M),D,\mathcal{H}\big)$ is an even spectral triple of dimension $n=2m$ with the standard Dirac
 operator $D$ and a grading $\Gamma$. We consider the usual double-sheet spectral triple, $\Big(C^\infty(M)\otimes \mathbb{C}^2,\widetilde{D},L^2\big(M,S(TM)\big)\Big)\oplus L^2\big(M,S(TM)\big)$ with the Dirac operator,
 \[
\widetilde{D}=\begin{bmatrix}
 D+c(V)&0\\
0&D+c(V) \\
\end{bmatrix}+\Gamma\otimes \begin{bmatrix}
0&\Phi\\
\Phi^*&0 \\
\end{bmatrix}=
\begin{bmatrix}
 D+c(V)&\Gamma\Phi\\
\Gamma\Phi^*&D+c(V) \\
\end{bmatrix},
\]
 where $\Phi\in \mathbb{C}.$ Then\[
\begin{bmatrix}
    \widetilde{D},&
    \begin{bmatrix}  f_1&0\\
0& f_2 \end{bmatrix}
\end{bmatrix}
=
\begin{bmatrix}
c(df_1)&\Phi(f_2-f_1)\Gamma\\
    \Phi^*(f_1-f_2)\Gamma& c(df_2)
\end{bmatrix},\]
where $f_1,f_2\in C^\infty(M).$ This is a noncommutative non-self-adjoint spectral triple.
\end{ex}

\section*{Declarations}
\textbf{Conflict of interest} The authors declare no conflicts of interest.

\section*{Acknowledgements}
This work was supported by the National Natural Science Foundation of China (No.11771070).

\end{document}